\documentclass[preprint,12pt]{elsarticle}

\usepackage{amssymb}
\usepackage{amsthm}
\usepackage{epsfig}
\usepackage[tbtags]{amsmath}
\usepackage{amsfonts,amssymb,mathrsfs,amscd}

\newtheorem{theorem}{Theorem}
\newtheorem{lemma}{Lemma}[section]
\newtheorem{assertion}{Assertion}
\newtheorem{corollary}{Corollary}
\newtheorem{remark}{Remark}

\journal{Topology and its Applications}

\begin{document}

\begin{frontmatter}

%% Title, authors and addresses

%% use the tnoteref command within \title for footnotes;
%% use the tnotetext command for theassociated footnote;
%% use the fnref command within \author or \address for footnotes;
%% use the fntext command for theassociated footnote;
%% use the corref command within \author for corresponding author footnotes;
%% use the cortext command for theassociated footnote;
%% use the ead command for the email address,
%% and the form \ead[url] for the home page:
%% \title{Title\tnoteref{label1}}
%% \tnotetext[label1]{}
%% \author{Name\corref{cor1}\fnref{label2}}
%% \ead{email address}
%% \ead[url]{home page}
%% \fntext[label2]{}
%% \cortext[cor1]{}
%% \address{Address\fnref{label3}}
%% \fntext[label3]{}

\title{Topological classification of Liouville foliations for the Kovalevskaya integrable case on the Lie algebra so(3, 1)}

%% use optional labels to link authors explicitly to addresses:
%% \author[label1,label2]{}
%% \address[label1]{}
%% \address[label2]{}

\author{V. Kibkalo}
\ead{slava.kibkalo@gmail.com}
\address{Moscow State University,\\ Faculty of Mechanics and Mathematics,  Russia}

\begin{abstract}
In this paper, we study the topology of the Liouville foliation of an analogue
of the Kovalevskaya integrable case on the Lie algebra $\mathrm{so} (3,1) $. The Fomenko-Zieschang invariants  (i.e., marked molecules) of a given foliation on each regular isoenergy surface were calculated.

Bibliography: 17 titles.
\end{abstract}

\begin{keyword}
Hamiltonian system \sep Kovalevskaya integrable case \sep Lie algebra $\mathrm{so} (3,1) $ \sep  Liouville foliation \sep  Fomenko--Zieschang invariant
%% keywords here, in the form: keyword \sep keyword

%% PACS codes here, in the form: \PACS code \sep code

%% MSC codes here, in the form: \MSC code \sep code
%% or \MSC[2008] code \sep code (2000 is the default)
\MSC[2010] 37J35\sep 70H33
\end{keyword}

\end{frontmatter}

%% \linenumbers

%% main text
%\section{}
%\label{}

\section{Introduction}

In this paper we study topological properties of an integrable generalization of the classical Kovalevskaya system in rigid body dynamics found by I.\,V.~Komarov  in \cite{Komarov81}. In that paper the classical Kovalevskaya top discovered by S. Kovalevskaya \cite{Kowalewski1889Acta12}, \cite{Kowalewski1889Acta14}, which is an integrable case of the Euler equations on the Lie algebra $\operatorname{e}(3)$, was included in a one-parameter family of integrable Hamiltonian systems on the pencil of Lie algebras $\operatorname{so}(4)-\operatorname{e}(3)-\operatorname{so}(3, 1)$. The Kovalevskaya top has been studied by many authors from various points of view, in particular its topology was studied in detail by M.\,P.~Kharlamov (see e.g. \cite{Kharlamov}, \cite{Kh88}). An important question of topological analyses of an integrable system is the study of its Liouville foliation. The topology of Liouville foliation for the Kovalevskaya top was completely described in \cite{BFR}. The same results for the Kovalevskaya case on  $\operatorname{so}(4)$ were obtained in \cite{Kozlov14}, \cite{LJM} and \cite{Sb}.

In this paper we generalize the results of \cite{BFR} to the analogue of the Kovaleskaya system on $\operatorname{so}(3, 1)$ using the results of M.\,P.~ Kharlamov, P.\,E.~Ryabov and A.\,Yu.~Savushkin \cite{Ryabov16}. More precisely, we calculate all the marks of the rough molecules found in \cite{Ryabov16}, thus obtaining all the Fomenko-Zieschang invariants of the system. All the necessary information about the Fomenko theory on the topological analysis of integrable Hamiltonian systems used in this paper can be found in \cite{BF}.

\section{Kovalevskaya case and its analogues}

The classical Kovalevskaya integrable case in rigid body dynamics was included by I.V. Komarov \cite{Komarov81} in a one-parameter family of dynamical systems on the pencil of Lie algebras $\operatorname{so}(4)-\operatorname{e}(3)-\operatorname{so}(3, 1)$. The Lie--Poisson brackets have the form \begin{equation} \label{Eq:Poisson_Lie_bracket_kappa} \{J_i, J_j\} = \varepsilon_{ijk}J_k, \quad \{J_i, x_j\} = \varepsilon_{ijk}x_k,
\quad \{x_i, x_j\} = \varkappa \varepsilon_{ijk}J_k, \end{equation}
where $\varepsilon_{ijk}$ is the permutation symbol and $\varkappa \in \mathbb{R}$. The cases $\varkappa>0, \varkappa =0$ and $\varkappa <0$ correspond to the Lie algebras $\operatorname{so}(4), \operatorname{e}(3)$ and $\operatorname{so}(3, 1)$ respectively.

These systems define integrable Hamiltonian systems with two degrees of freedom on every regular level surface  $M^4_{a, b} = \left\{ f_1 =a, f_2 = b\right\} $ of the Casimir functions of the brackets \eqref{Eq:Poisson_Lie_bracket_kappa}:
\begin{equation}\label{Eq:Kasimirs}  f_1 = (x_1^2 + x_2^2 + x_3^2) + \varkappa (J_1^2 +J_2^2 +J_3^2),
\qquad  f_2 = x_1 J_1 + x_2 J_2 +x_3 J_3.
\end{equation} The Hamiltonian $H$ and the first integral $K$ have the form \begin{equation}
    \label{Eq:Hamiltonian}
    H = J_1^2 + J_2^2 + 2J_3^2 + 2 c_1 x_1,
\end{equation}
\begin{equation}\label{Eq:First_Integral} K = (J_1^2 - J_2^2-2c_1 x_1 + \varkappa
c_1^2)^2 + (2J_1 J_2 - 2 c_1 x_2)^2.
\end{equation} Here $c_1$ is an arbitrary constant. We can assume that $c_1 = 1$ and $\varkappa \in \{-1, 0, 1\}$ (see e.g. \cite{Kozlov14}).

We will derive some information about the Kovalevskaya case on $\operatorname{so}(3, 1)$ from some other integrable systems on $\operatorname{e}(3)$ studied earlier. Essentially, we use a Poisson diffeomorphism between the open subsets of $\operatorname{e}(3)^*$ and $\operatorname{so}(3,1)^*$ described in \cite{Tsiganov03} to identify the Kovalevskaya case on $\operatorname{so}(3, 1)$ with the Kovalevskaya--Sokolov case on $\operatorname{e}(3)$ studied in \cite{Ryabov16}. This diffeomorphism from \cite{Tsiganov03} can be described as follows.

\begin{theorem}[\cite{Tsiganov03}] \label{A:PoissonMap} Consider $\mathbb{R}^6(\mathbf{J}, \mathbf{y})$ with the Lie-Poisson bracket for $\operatorname{e}(3)$: \begin{equation} \label{Eq:Poisson_Lie_bracket_E3} \{J_i, J_j\} = \varepsilon_{ijk}J_k, \quad \{J_i, y_j\} = \varepsilon_{ijk}y_k, \quad \{y_i, y_j\} = 0. \end{equation} The Casimir functions of this bracket are $g_1 = \mathbf{y}^2$ and $g_2 = \left(\mathbf{J}, \mathbf{y} \right)$.  Fix $\varkappa< 0$, $\alpha \not =0$. Then in the new coordinates  $(\mathbf{J}, \mathbf{x})$ on $\mathbb{R}^6(\mathbf{J}, \mathbf{y}) \backslash \left\{g_1 =0 \right\} $, where \begin{equation} \label{Eq:Tsig_PoissDiffCoord} \mathbf{x} = \alpha \mathbf{y} + \sqrt{\frac{-\varkappa}{\mathbf{y}^2}} \mathbf{y} \times \mathbf{J}, \end{equation}   the Poisson bracket \eqref{Eq:Poisson_Lie_bracket_E3} takes the form \eqref{Eq:Poisson_Lie_bracket_kappa} of the Lie-Poisson bracket  for $\operatorname{so}(3,1)$. Thus  \begin{equation} \label{Eq:Tsig_PoissDiff} (\mathbf{J}, \mathbf{y}) \to (\mathbf{J}, \mathbf{x}) \end{equation} is a Poisson diffeomorphism between $\operatorname{e}(3)^* \backslash  \left\{ g_1=0, g_2=0 \right\}$ and $\operatorname{so}(3,1)^* \backslash  \left\{ f_1\leq 0, f_2=0 \right\}$. A non-singular orbit $M^4_{a, b} = \left\{ g_1 = a, g_2 = b\right\} $ of $\operatorname{e}(3)^*$ is identified with the orbit $M^4_{\alpha^2 a+ \varkappa \frac{b^2}{a}, \alpha b}$ of $\operatorname{so}(3,1)^*$.   \end{theorem}

The Hamiltonian \eqref{Eq:Hamiltonian} takes the form \eqref{Eq:Ham_Sok_linear} of the Kovalevskaya--Sokolov case from \cite{Ryabov16} in the coordinates $(\mathbf{J}, \mathbf{y})$ given by \eqref{Eq:Tsig_PoissDiffCoord}
\begin{equation} \label{Eq:Ham_Sok_linear} H_1 = J_1^2 +  J_2^2 + 2 J_3^2 +2 c_1 y_1 +2 c_2 (y_2 J_3 - y_3 J_2) \end{equation}
for some new constants $c_1, c_2$.
As it was noted in \cite{Tsiganov03} the Kovalevskaya-Sokolov case can be included in a wider family of integrable systems with Hamiltonian \eqref{Eq:Ham_Sok_old}  \begin{equation}\label{Eq:Ham_Sok_old} H_\varkappa = J_1^2 + J_2^2 + 2 J_3^2 +2 c_1 y_1 - 2 c_2 J_3 y_2 - c_2^2 y_3^2 + 2 c_3 (J_3 + c_2 y_2)\end{equation}
 after the following change of coordinates
\[  J_2 \to  J_2 - c_2 y_3, \quad J_3 \to  J_3 + c_2 y_2. \]

The family of Hamiltonians \eqref{Eq:Ham_Sok_old} is  integrable for all values of parameters $c_1, c_2, c_3$ and for all values $a, b$ of Casimir functions $f_1, f_2$ and the parameter $\varkappa \in \mathbb{R}$ of the brackets \eqref{Eq:Poisson_Lie_bracket_kappa} (see \cite{Tsiganov03}). It includes the well-known Hamiltonians of the Kovalevskaya case (for  $c_2 = c_3 =0$), the Kovalevskaya--Yehia case (for $c_2 =0$) and the Sokolov case (for $c_1 = c_3 =0$). We will use some information about the Fomenko--Zieschang invariants for the Sokolov case from  \cite{Morozov04}.

\begin{remark} Note that the first integral $\widetilde{K_s}$ for the the family \eqref{Eq:Ham_Sok_linear} for $c_1=0$ from \cite{Ryabov16} is differs from the first integral of the Sokolov case $K_s$, written in \cite{Sokolov01} and \cite{Ryabov03}, but they are related by the formula \begin{equation} \label{Eq:Kov_type_for_Sok}
K_s = -4 \widetilde{K}_s + \cfrac{(2 H_s - c_2^2 f_1)^2}{4} - c_2^2 f_2^2.
\end{equation} \end{remark}

\section{Bifurcation diagrams of the Kovalevskaya case on $\operatorname{so}(3, 1)$}

Let us start by constructing the bifurcation diagrams for the Kovalevskaya case on $\operatorname{so}(3, 1)$.  It can be done similarly to \cite{Kozlov14}, where the case of $\operatorname{so}(4)$ was considered. Note that the majority of formulae from \cite{Kozlov14} still holds for $\varkappa<0$.  We deal separately with cases $b=0$ and $b\not =0$.

\subsection{Case of zero area integral ($b=0$)}

Bifurcation diagram $\Sigma$ of the Kovalevskaya integrable case is contained in the union of bifurcation curves on $Ohk$ plane for $\varkappa \ne 0$ given by the same formulae as in \cite{Kozlov14}. We say that two bifurcation diagrams are \textit{structurally different} if they contain different set of arcs of these curves or have different bifurcations in the $\mathfrak{F} = (H, K)$-preimage of these arcs.

\begin{theorem}
A bifurcation diagram of the Kovalevskaya integrable case on $so(3, 1)$ ($\varkappa <0$) for the zero area integral $f_2$ (i.e. $b =0$) is determined by the value $a$ of the Casimir function $f_1$. Diagrams for the following six intervals XII-XVII of the $Oa$ axis are shown in Fig. \ref{Fig:bif_diagrams_so31}. Here $4a_0: = \varkappa^2 c_1^2$. %
\[Oa: \quad -\infty,\, \mathrm{XII},\, -4a_0,\, \mathrm{XIII},\, -a_0,\, \mathrm{XIV},\, 0,\, \mathrm{XV},\, a_0,\,  \mathrm{XVI},\, 4a_0,\,  \mathrm{XVII},\, +\infty.\]
\end{theorem}

\begin{figure}
\minipage{1.1\textwidth}
\includegraphics[width=\linewidth]{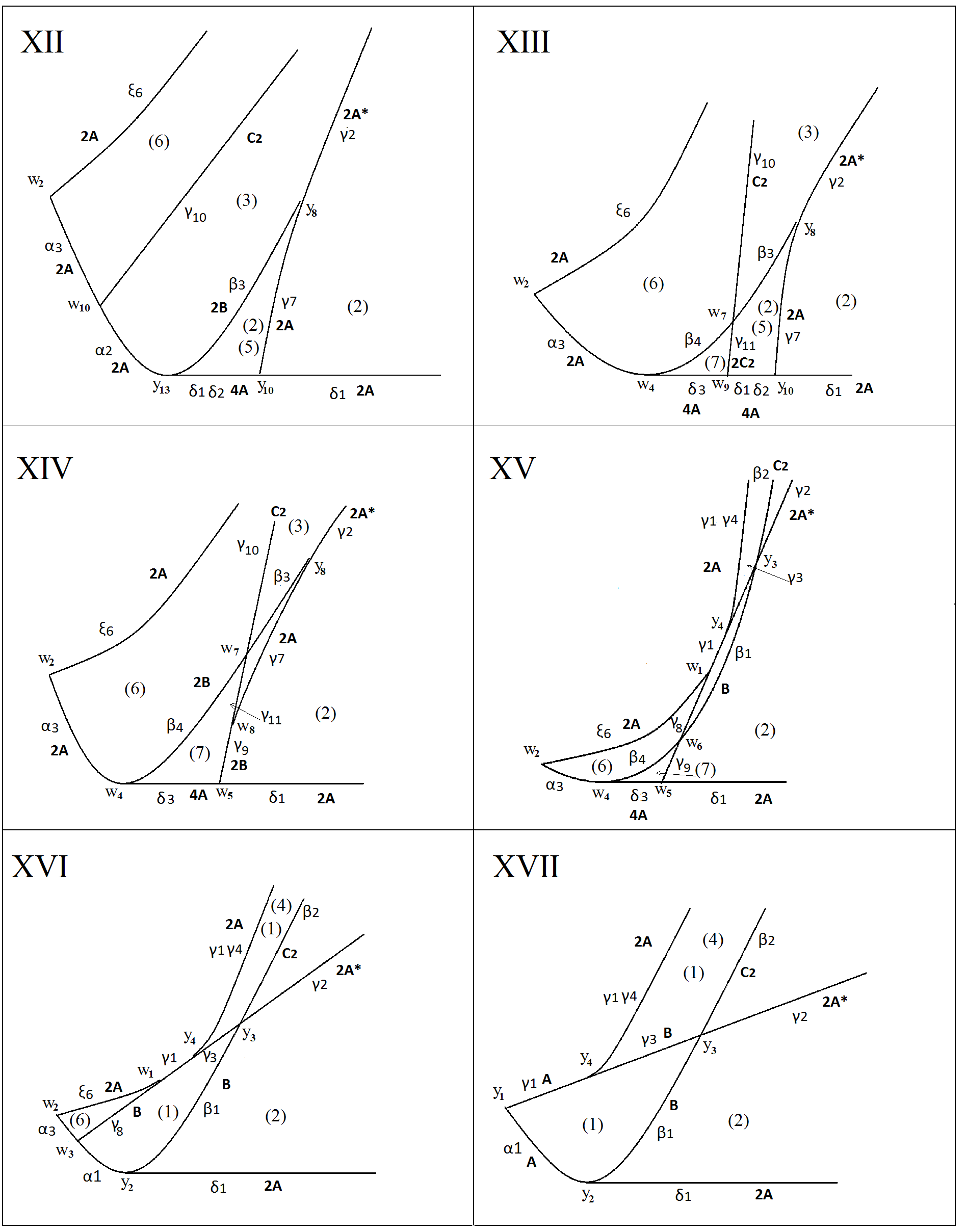}
   \caption{Bifurcation diagrams of Kovalevskaya $so(3, 1)$ case for $b =0$} \label{Fig:bif_diagrams_so31}
\endminipage
\end{figure}

Let us denote and enumerate the arcs of $\Sigma$ by the same way as in \cite{BFR}, \cite{LJM}. The letter in its notation is determined by the bifurcation curve and numbers should be different. All necessary information about the ``new'' arcs (i.e. arcs without analogues in the Kovalevskaya case on $\mathrm{e}(3)$) is contained in the Table \ref{Tab:new_arcs}. Let us call other arcs ``old'' arcs.  We also specify the Liouville tori families and the number of tori in the preimage of a point above and under an arc.

\begin{table}[h]
\centering
\begin{tabular}[t]{|c|c|c|c|c|c|c|c|c|c|c|c|c|c|}
\hline
  new & atom & higher & lower &  arc's endpoints & reg & $b =0$\\
\hline
 $\xi_6$ & $2A$ & $2T^2$ & $\oslash$   &  $w_2; \{w_1, \infty\}$ & V-XI & XII-XVI \\
\hline
 $\alpha_3$ & $2A$ & $\oslash$ & $2T^2$  & $w_2; \{w_3, w_4, w_{10}\}$ & V-XI & XII-XVI \\
\hline
 $\beta_4$ & $2B$ & $4T^2$ & $2T^2$  &  $w_4; \{w_6, w_7\}$ & IX-XI & XIII-XV \\
\hline
 $\delta_3$ & $4A$ & $\oslash$ & $4T^2$  & $w_4; \{w_5, w_9\}$  & IX-XI & XIII-XV \\
\hline
 $\gamma_8$ & $B$ & $T^2$ & $2T^2$  & $w_1; \{w_3, w_6\}$ & V-XI & XV-XVI \\
\hline
 $\gamma_9$ & $2B$ & 4T & 2T  & $w_5; \{w_6, w_8\}$ & IX-XI & XIV-XV \\
\hline
 $\gamma_{10}$ & $C_2$ & 2T & 2T  & $\{w_{10}, w_7\}; \infty$ & - & XIII-XIV \\
\hline
 $\gamma_{11}$ & $2C_2$ & 4T & 4T  & $w_8; w_7$ & - & XIII-XIV \\
\hline
\end{tabular}
\caption{New arcs of $\Sigma$ for the Kovalevskaya case,  $\varkappa <0$}
\label{Tab:new_arcs}
\end{table}

The preimage of singular points of $\Sigma$ (i.e. points of intersection, tangency and return points of bifurcation curves) contains all critical points of $\mathfrak{F}$ with $\mathrm{rk} = 0$ and circles that consist of degenerate critical points of $\mathrm{rk} = 1$.

\begin{remark} All critical points of $\mathrm{rk} = 1$ except for the preimages of $w_1, w_4, w_8$ are nondegenerate, see Assertions 6-8 in \cite{Kozlov14}.
\end{remark}

We use notations $y_i$ from \cite{Kozlov14} for an ``old'' point and $w_i$ for a ``new'' one according to Table \ref{Tab:new_points}. In that table we also specify the rank, number of orbits in the preimage and the loop molecule of the singularity. We also include in Table \ref{Tab:new_points} the notations of these points from \cite{Ryabov16} and classes from \cite{Kozlov14}.

Nondegeneracy of critical points of $\mathrm{rk} =0$ can be easily checked as in~\cite{Kozlov14}.
\begin{assertion}
All critical points of $\mathrm{rk} = 0$ are nondegenerate in an orbit $M^4_{a, 0}$ for $a \notin \{\pm \varkappa^2 c_1^2, \pm \cfrac{\varkappa^2 c_1^2}{4}, 0\}$,\, $\varkappa <0$. Singular points $w_2, w_3, w_5, w_6, w_7, w_9, w_{10}$ are their $\mathfrak{F}$-images in $\Sigma_{a, 0}$. Nondegenerate singularities of $\mathrm{rk} = 0$ are represented in Table \ref{Tab:new_points} as a product of two foliated 2-atoms (see \cite{BF}). We also specify their regions and intervals of $Oa$ for which a point $w_i$ belongs to $\Sigma_{a, 0}$.
\end{assertion}
\begin{table}[h]
\centering
\begin{tabular}[t]{|c|c|c||c|c|c|c|c|c|c|c|c|c|c|}
\hline
name & \cite{Ryabov16} & \cite{Kozlov14} & $rk$ & \# & loop molecule. & $b \ne 0$ & $Oa$\\
\hline
$w_1$ & $\gamma_{ij}$ & $rt$ & $1$ & 1 & ell. pitchfork   &  V-XI & XV-XVI \\
\hline
$w_2$ & $\delta_4$ & $h_{int}$ & $0$  & 2 & 2 $A \times A$   &   V-XI & XI-XVI \\
\hline
$w_3$ & $\delta_{12}$ & $+l$  & $0$ &  1 & $A \times B$  &   V-VIII & XVI \\
\hline
$w_4$ & $\gamma_{ij}$ & $h_l$ & $1$ &  2 & 2 ell. pitchfork  &   IX-XI & XIII-XV \\
\hline
$w_5$ & $\delta_{36}$ & $r_3$ & $0$ &  2 & 2 $A \times B$  &   IX-XI & XIV-XV \\
\hline
$w_6$ & $\delta_{13}$ & $+l$ & $0$ &  1 & $B \times B$  &   IX-XI & XV \\
\hline
$w_7$ & - & $+l$ & $0$  & 2 & $B \times C_2$ &   - & XIII-XIV \\
\hline
$w_8$ & - & $z_{cusp}$ & $1$ & 2 & 2 hyp. pitchfork   &   - & XIV \\
\hline
$w_9$ & - & $r_{3, l}$ & $0$ & 4 & 2 $A \times C_2$ &   - & XIII \\
\hline
$w_{10}$ & - & $+l$  & $0$ & 2 & $A \times C_2$ &   - & XII \\
\hline
\end{tabular}
\caption{New singular points of $\Sigma$ for the Kovalevskaya case, $\varkappa <0$}
\label{Tab:new_points}
\end{table}

\begin{remark}
Loop molecules of singular points $w_1, w_4, w_8$ conicide with loop molecules of typical degenerate singularities of the $\mathrm{rk} = 1$: elliptic pitchfork ($w_1, w_4$) and hyperbolic pitchfork ($w_8$), see \cite{BFR}.
\end{remark}

\begin{proof}
1)  The bifurcation diagram $\Sigma_{a, 0}$,  i.e. the image of a critical set in an orbit $M^4_{a, 0}$, can be obtained from diagrams $\Sigma_{a_n, b_n}$  for orbits $M^4_{a_n, b_n}$ constructed in \cite{Ryabov16} by passing to the limit $(a_n, b_n) \rightarrow (a, 0)$. This follows from the compactness of the following set $A$ for fixed $\varepsilon <0$ and $h_0$: \[A = \{(\textbf{J}, \textbf{x})\,|\,\, H = h < h_0, \,\,f_1  = a < -\varepsilon <0\}.\] In particular, if a sequence of singular points $x_n \in \Sigma_{a_n, b_n}$ has a finite limit $x$ when $(a_n, b_n) \rightarrow (a, 0)$, then $x \in \Sigma_{a, 0}$.

2) For each value $a$ of $f_1$ it is easy to determine the number and type of singular points of rank  $0$, based on Assertions 15-17 of \cite{Kozlov14}. At the same time, their nondegeneracy was checked and the types were determined: center-center, center-saddle, and saddle-saddle.

3) Since the critical set is closed, and every isoenergy regular manifold $Q^3_{a, b, h}$ is compact, the critical points of rank  $1$ in the preimage of an every non-singular point of bifurcation diagram $\Sigma$ form one circle or several circles.

The number of critical circles in the preimage of arc's points of a diagram and the number of tori in areas of the $Ohk$  plane are uniquely determined from the number and type of singular points of rank $0$ in the preimages of the singular points $w_i$ of diagrams $\Sigma_{a, 0}$.

4) The only remaining question is the type of bifurcation $X$  in the preimage of the arc $\gamma_{10}$.  Since the singularity $w_{10}$ is of center-saddle type with two critical points of rank $0$ and the singular fiber of the atom $X$ is connected, either $X = C_2$ or $X = D_2$.

After a perturbation the atom $D_1$ splits into two atoms $B$  and every level $K=k$ between them must have three connected components. At the same time, the Kovalevskaya system has only one component at a such level. Thus $X = C_2$. \end{proof}
\begin{corollary} Therefore, a decomposition of the saddle-saddle singularity $B \times C_2$ into two singularities $B \times B$  described in \cite{OshTuz18} is realized in the Kovalevskaya system on $\operatorname{so}(3, 1)$. \end{corollary}

\subsection{Connections with the Kovalevskaya--Sokolov case}

If $b\not =0$, then the Kovalevskaya case can be identified with the Kovalevskaya-Sokolov case, for which the bifurcation diagrams were constructed in \cite{Ryabov16}. We briefly recall some facts and notations from \cite{Ryabov16}. There every orbit $M^4$ was characterised by pair of numbers $(\zeta_{*}, l_{*})$.
There $l = \langle M, \alpha \rangle/2, \, \hat{a}^2 = \langle \alpha, \alpha \rangle$, where $\alpha = (\alpha_1, \alpha_2, \alpha_3)$ and $M = (M_1, M_2, M_3)$ are the coordinate and the velocity vectors in $\mathbb{R}^6$ for the Kovalevskaya--Sokolov integrable case. $l_{*}$ are $\zeta_{*}$ are the following:
\[l_{*}^3 = \cfrac{2 l^2}{\varepsilon_0 \hat{a}^3}, \quad \zeta_{*} = \cfrac{\hat{a} \zeta }{\varepsilon_0} = \cfrac{\hat{a} \varepsilon_1^2}{\varepsilon_0}.\]

The set $\left\{ l_{*} >0, \zeta_{*} >0\right\}$ is separated by five curves $\theta_1, \dots, \theta_5$ into 11 regions $1 \dots 11$, see Fig. \ref{Fig:regions_Kov_Sok}. Points from the same region have structurally equivalent bifurcation diagrams $\Sigma$ in the Kovalevskaya--Sokolov case on $\mathrm{e}(3)$.

Separating curves $f_r, f_k, f_m$ and two intervals $f_t'$ and $f_t''$ of the curve $f_t$ except for the points of the axis $Oa$ (i.e. without the endpoints of intervals XII-XVII) on the plane $Oab$ for the Kovalevskaya case on $so(3, 1)$  are the images of curves $\theta_1, \dots, \theta_5$ under the Poisson map from  Assertion~\ref{A:PoissonMap} (we specify the correspondence between these curves in Table~\ref{Tab:sep_set_corr}).

\begin{table}[h]
\centering
\begin{tabular}[t]{|c|c|c|c|c|c|c|c|c|c|c|c|c|c|}
\hline
Kovalevskaya--Sokolov case & $\theta_1$  & $\theta_2$  & $\theta_3$ & $\theta_4$ & $\theta_5$ \\
\hline
Kovalevskaya $so(3,1)$ case & $f_k$  & $f_t'$  & $f_t''$ & $f_r$ & $f_m$ \\
\hline
 Points on the axis $Oa$ & $-a_0$  & $-4a_0$ & $-4a_0, a_0$ & $0$ & $4a_0$ \\
\hline
\end{tabular}
\caption{Correspondence of separating curves}
\label{Tab:sep_set_corr}
\end{table}

Regions I-XI of Kovalevskaya case on $\mathrm{so}(3, 1)$ are the Poisson map image of regions 1-11 of Kovalevskaya-Sokolov case. Regions 1, 2, 3 and 4 for the latter case (i.e. regions I, II, III, IV for the former case respectively)  have analogues in the classical Kovalevskaya case ($\varkappa =0$). Let us enumerate other regions in the same way preserving the notation from \cite{BFR}.

\begin{figure}[]
\minipage{0.56\textwidth}
\includegraphics[width=\linewidth]{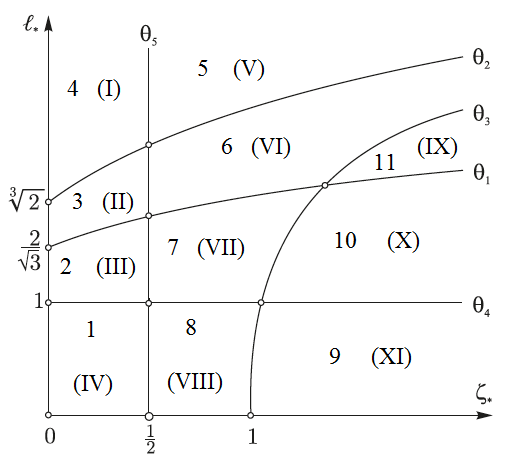}
   \caption{Kovalevskaya--Sokolov case} \label{Fig:regions_Kov_Sok}
\endminipage
\hspace{0.04\textwidth}
\minipage{0.40\textwidth}
\includegraphics[width=\linewidth]{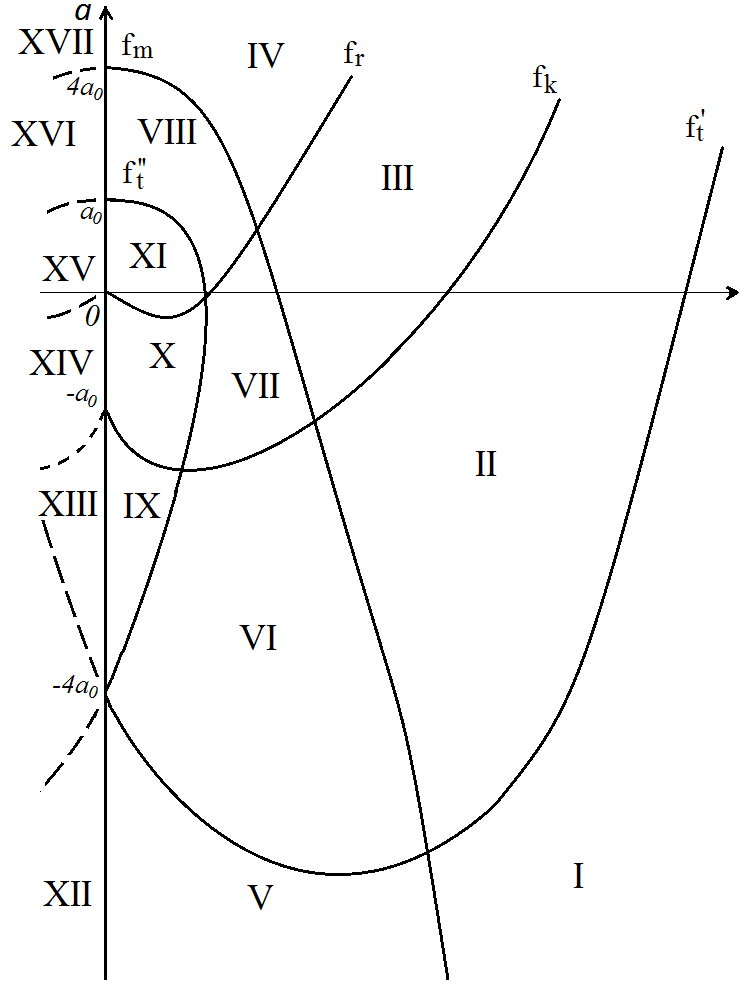}
   \caption{Kovalevskaya case on $so(3, 1)$} \label{Fig:regions_Kov_so31}
\endminipage
\end{figure}

\section{Liouville analysis of Kovalevskaya case on Lie algebra $\mathrm{so}(3, 1)$}

In this section we will calculate all the Fomenko-Zieschang invariants for every 3-dimensional regular isoenergy submanifold $Q^3$ of the Kovalevskaya integrable  case on $\mathrm{so}(3, 1)$. The list of rough molecules (see \cite{BF}) was constructed in \cite{Ryabov16}. It remains to find the marks of these molecules. We find them similarly to \cite{LJM} by explecitly expressing the admissible coordinate system (see \cite{BF}) for every ``new'' arc  of the bifurcation diagram through the uniquely defined $\lambda$-cycles of bifuractions.

We will use the information about some previously studied integrable cases, namely  the classical Kovalevskaya case and the Sokolov integrable case studied in \cite{Ryabov03} and\cite{Morozov04}. The first case is the limit of the Kovalevskaya case on $so(3, 1)$ when  $\varkappa \rightarrow -0$. And the second one is the limit of the Kovalevskaya--Sokolov case when $\varepsilon_0 \longrightarrow 0$  (see \cite{Ryabov16}).

\begin{lemma}
Regions 5, 10, 11 of the plane $(\zeta_{*}, l_{*})$ of parameters for the Kovalevskaya--Sokolov case are preserved after passing to limit $\varepsilon_0 \rightarrow +0$.
\end{lemma}

\begin{proof} Parabolas $l = \hat{a}^2 \varepsilon_1$ and $l =\hat{a}^2 \varepsilon_1/2$ are the limits of the separating curves $\theta_1$ and  $\theta_2, \theta_3$ respectively. Axes $O\hat{a}$ and $Ol$ are the limits of the curves $\theta_4$ and $\theta_5$ respectively.
\end{proof}

Some arcs that appear in both limiting cases have different notations. We will start with notations for the classical Kovalevskaya case from \cite{BFR} and will explain correspondences for other arcs.

\begin{theorem} \label{T:AdmCoord}
Admissible coordinate systems for ``new'' arcs $\alpha_3, \gamma_8, \gamma_9, \gamma_{10}, \gamma_{11}, \beta_4, \delta_3, \xi_6$ are expressed via the uniquely defined $\lambda$-cycles as follows:
\[
\alpha_3:\,\, \cfrac{(\lambda_{\alpha_3}, \lambda_{\beta_4})_{(6)}}{\oslash}\,, \quad
\beta_4:\,\, \cfrac{(\lambda_{\beta_4}, - \lambda_{\gamma_8})_{(6)}}{(\lambda_{\beta_4}, \lambda_{\gamma_9})_{(7)}}\,, \quad
\gamma_8:\,\, \cfrac{(\lambda_{\gamma_8}, \lambda_{\beta_4})_{(6)}}{(\lambda_{\gamma_8}, -\lambda_{\beta_1})_{(1)}}\,, \quad
\gamma_9:\,\, \cfrac{(\lambda_{\gamma_9}, \lambda_{\beta_4})_{(7)}}{(\lambda_{\gamma_9}, -\lambda_{\beta_1})_{(2)}}\,,
\]
\[
\gamma_{10}:\,\, \cfrac{(\lambda_{\gamma_{10}}, \lambda_{\beta_4})_{(6)}}{(\lambda_{\gamma_{10}}, -\lambda_{\beta_3})_{(3)}}\,, \quad
\gamma_{11}:\,\, \cfrac{(\lambda_{\gamma_{11}}, \lambda_{\beta_4})_{(7)}}{(\lambda_{\gamma_{11}}, -\lambda_{\beta_3})_{(2, 5)}}\,, \quad
\delta_3:\,\, \cfrac{(\lambda_{\delta_3}, \lambda_{\beta_4})_{(7)}}{\oslash}\,, \quad
\xi_6:\,\, \cfrac{\oslash}{(\lambda_{\xi_6}, \lambda_{\gamma_5})_{(6)}}\,.
\] The location of these $\lambda$-cycles on the fundamental lattice of Liouville tori is shown in Fig. \ref{Fig:cycles}.

\end{theorem}
\begin{figure}
\minipage{1.\textwidth}
\includegraphics[width=\linewidth]{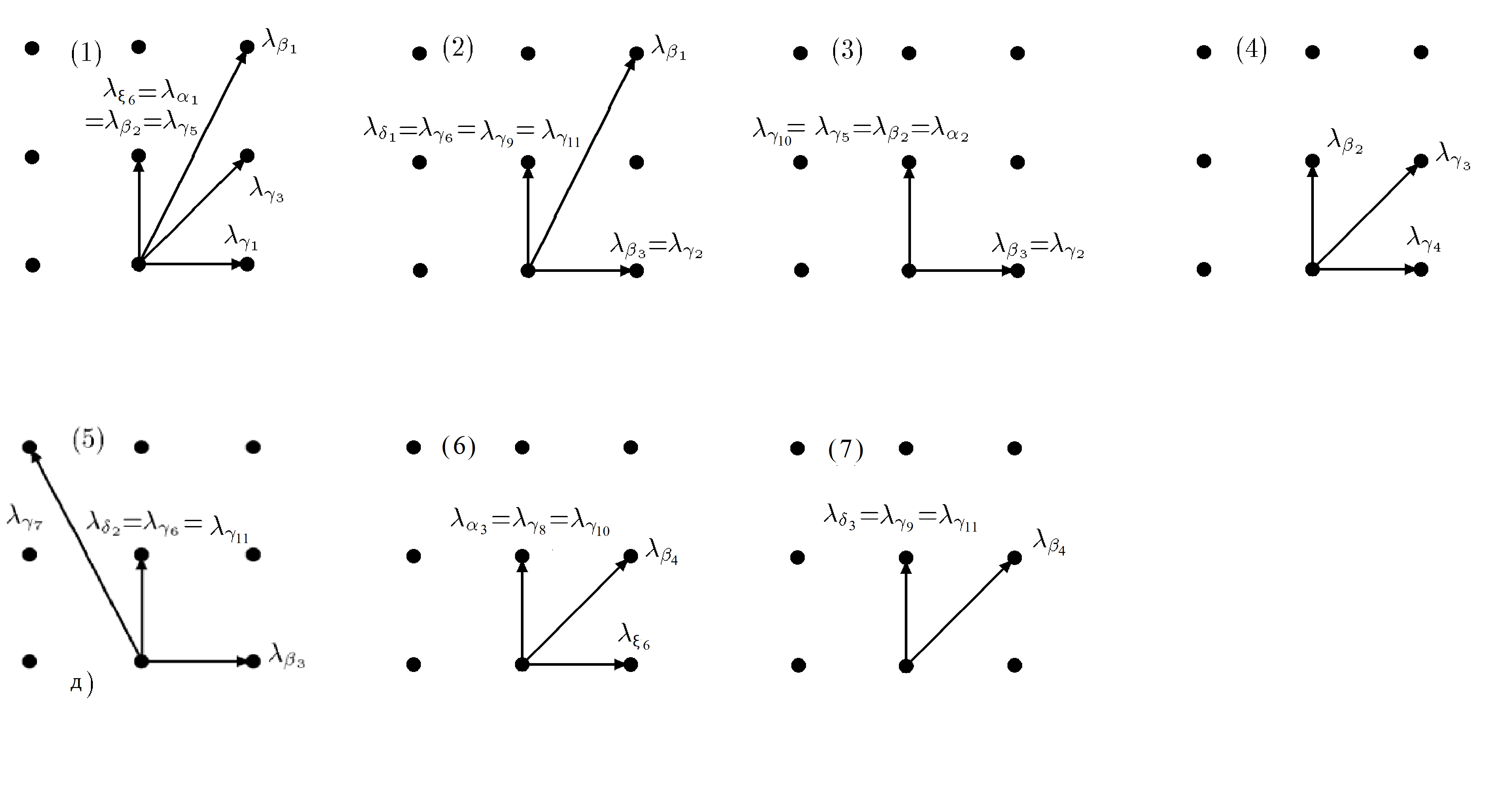}
   \caption{Uniquely defined $\lambda$-cycles for families (1)-(7) of Liouville tori}
   \label{Fig:cycles}
\endminipage
\end{figure}
\begin{remark}

 The cycle $\lambda_{\beta_4}$ on tori of the family (6) (i.e. $\lambda_{\beta_1}$ on family I of tori in \cite{Morozov04}) and $\lambda_{\beta_4}$ on tori of the family (7) (i.e. $\lambda_{\beta_4}$ on family V of tori in \cite{Morozov04}) are expressed similarly by a choose of basis of the lattice. \end{remark}

\begin{proof}

1. If an arc  of $\Sigma$ and the type of bifurcation do not change when passing to a limiting case, then the corresponding admissible coordinate system also remains the same.

Therefore the admissible coordinate systems for the arcs $\alpha_3, \beta_4,  \gamma_{10}, \gamma_{11}, \delta_3, \xi_6$ are known from \cite{Morozov04}.

2. Singularities $w_{3}$ and $y_{12}$ are center-saddle $A \times B$ singularities, thus $\lambda_{\xi_1} = \pm \lambda_{\alpha_1}$ and $\lambda_{\gamma_5} = \pm \lambda_{\alpha_1}$, i.e. $\lambda_{\xi_1} = \pm \lambda_{\gamma_5}$

The symmetry $(\alpha, x) \longrightarrow (-\alpha, -x)$ maps tori of families (2), (3), (5), (6) one to another and preserves $\mathrm{sgrad}\,H$ in the Sokolov case. It means that the direction of $\mathrm{sgrad}\,H$ coincides on two critical circles of the $C_2$ atom for the arcs $\gamma_{10}, \gamma_{11}$.

Thus $\lambda_{\xi_1} = \lambda_{\gamma_5}$ and $\lambda_{\xi_2} = \lambda_{\gamma_6}$ for the resulting foliations $B (\xi_1) - B (\gamma_5)$ and  $B (\gamma_6) - B (\xi_2)$ after the perturbation of atoms  $C_2 (\xi_3)$ and $2C_2 (\xi_4)$.

3. Admissible coordinate systems for two arcs $\xi_1, \xi_2$ can be calculated by considering  a saddle-saddle singularity $w_6$ of the type $B \times B$ the same way as the singularity $U_2$ ($y_3$) was studied in \cite{BFR}.

\end{proof}

\begin{remark} Using Theorem~\ref{T:AdmCoord} we can calculate numerical marks $r, \varepsilon, n$ of the Fomenko--Zieschang invariants by well-known formulae, see \cite{BF}. (Molecules are directed by the acsending of the integral $K$.)
\end{remark}

We use notations from \cite{Ryabov16} for classes of regular foliations on isoenergy $Q^3$.

\begin{theorem}
1. For $b \ne 0$ all 25 classes of regular foliations on isoenergy $Q^3$ in the Kovalevskaya system on $\mathrm{so}(3, 1)$ are non-equivalent, i.e. have different Fomenko--Zieschang invariants:
\begin{itemize}
\item 10 classes $\mathbb{A}_1, \mathbb{A}_2, \mathbb{A}_3, \mathbb{B}_1, \mathbb{B}_2, \mathbb{C}_1, \mathbb{C}_2, \mathbb{B}_3, \mathbb{D}_1, \mathbb{A}_4$ coincide with classes $A$-$J$ respectively for the Kovalevskaya case on $\mathrm{e}(3)$,
\item 3 classes $\mathbb{E}_1, \mathbb{E}_2, \mathbb{F}_1$ coincide with classes $A, B, C$ of Sokolov case,
\item 6 classes $\mathbb{B}_4, \mathbb{B}_5, \mathbb{C}_3, \mathbb{C}_4, \mathbb{F}_2, \mathbb{G}$ can be obtained from classes $E, F, I, H, D, G$ of Sokolov case by a typical perturbation $C_2 \longrightarrow B \cfrac{r = \infty}{\varepsilon = 1} B$,
\item Fomenko-Zieschang invariants of classes $\mathbb{A}_5, \mathbb{A}_6, \mathbb{A}_7, \mathbb{A}_8, \mathbb{B}_6, \mathbb{D}_2$ are shown in Fig.~\ref{Fig:invariants}.
\end{itemize}

2. In the case of $b =0$ every foliation belongs to one of classes $A$-$I$ of Sokolov case, $A$-$D$ of Kovalevskaya case or $\mathbb{A}_5, \mathbb{A}_6$.

3. Foliations $\mathbb{A}_5$ and $\mathbb{A}_5$ belong to the classes $L_{26}$ and $L_7$ of Kovalevskaya integrable case on $\mathrm{so(4)}$.
\end{theorem}

\begin{remark}
Two foliations of the same Liouville class have the same structure of closures of almost all their trajectories.
\end{remark}

\begin{figure}
\minipage{0.75\textwidth}
\includegraphics[width=\linewidth]{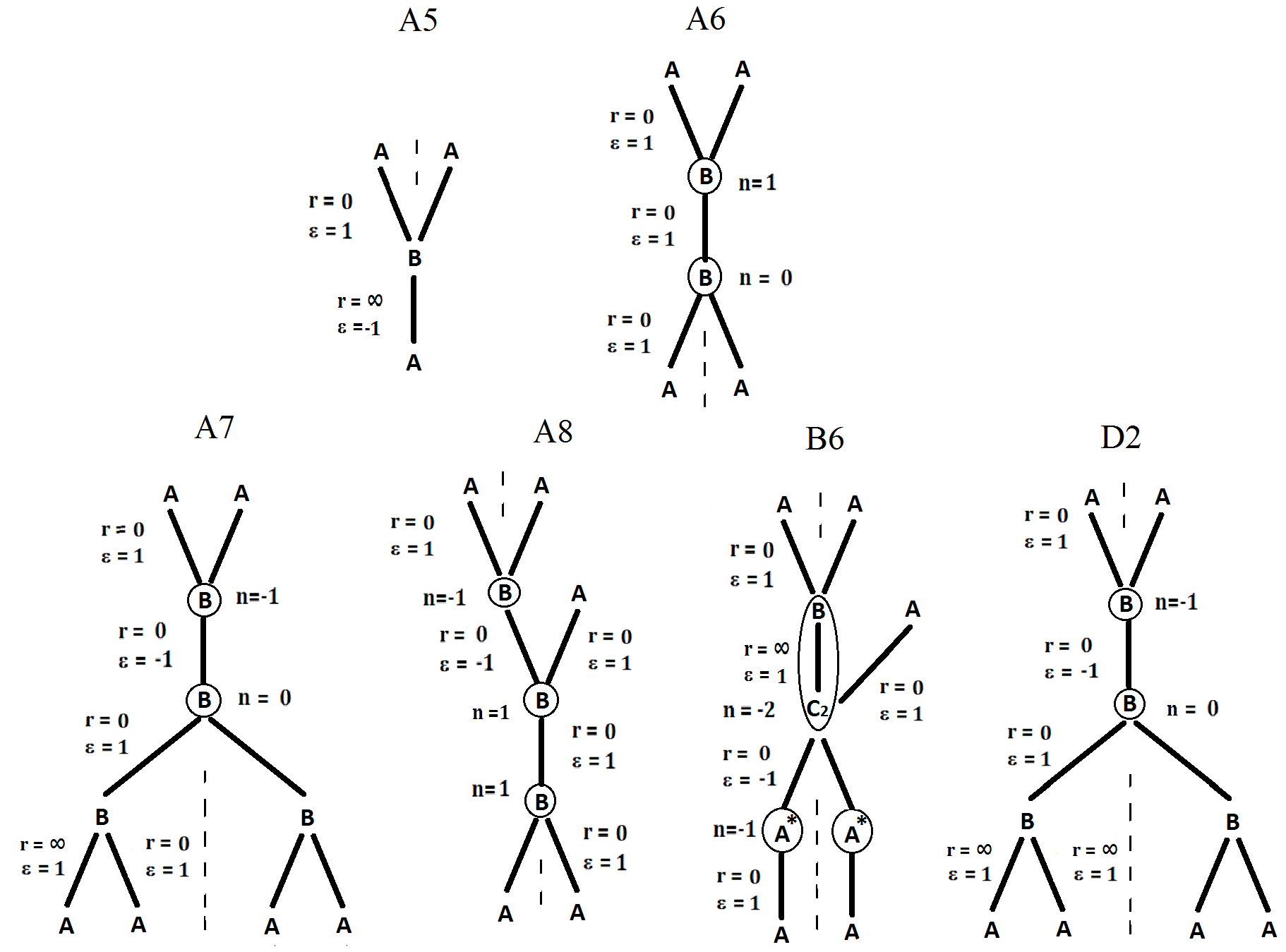}
   \caption{New invariants of the Kovalevskaya case on $so(3, 1)$}
   \label{Fig:invariants}
\endminipage
\end{figure}

\section{Acknowledgements}

This work was supported by the Russian Science Foundation (project no. 17-11-01303). Author is grateful to I.~Kozlov for fruitful discussions and help.

%% The Appendices part is started with the command \appendix;
%% appendix sections are then done as normal sections
%% \appendix

%% \section{}
%% \label{}

%% If you have bibdatabase file and want bibtex to generate the
%% bibitems, please use
%%
%%  \bibliographystyle{elsarticle-num}
%%  \bibliography{<your bibdatabase>}

%% else use the following coding to input the bibitems directly in the
%% TeX file.

\end{document}